\documentclass{amsart}

\usepackage{amsmath}
\usepackage{paralist}
\usepackage{graphics} 
\usepackage{epsfig} 
\usepackage{graphicx}  
\usepackage{epstopdf}
\usepackage[colorlinks=true]{hyperref}
\hypersetup{urlcolor=blue, citecolor=red}

\usepackage[utf8]{inputenc}
\usepackage{amsfonts}
\usepackage{amsmath}
\usepackage{graphicx}
\usepackage{amsthm}
\usepackage{pdfpages}
\usepackage{MnSymbol}

\usepackage[bottom = 4cm]{geometry}

\usepackage{bbm}
\usepackage{cite}
\usepackage{mathtools}
\usepackage{afterpage}
\usepackage{float}

\usepackage[pagewise]{lineno}

\numberwithin{equation}{section}

\newtheorem{theorem}{Theorem}[section]

\newtheorem{lemma}[theorem]{Lemma}
\newtheorem{proposition}[theorem]{Proposition}

\newtheorem{assumption}[theorem]{Assumption}

\theoremstyle{definition}
\newtheorem{definition}[theorem]{Definition}

\theoremstyle{remark}
\newtheorem{remark}[theorem]{Remark}

\numberwithin{figure}{section}

\newcommand{\R}{\mathbb R}

\newcommand{\N}{\mathbb N}
\newcommand{\supp}{\text{supp}}

\newcommand{\dy}{{\rm d} y}
\newcommand{\dx}{{\rm d} x}

\newcommand{\su}{\overline{u}}

\newcommand{\DN}{N}
\newcommand{\OS}{\Omega \subset \R^\DN}

\newenvironment{nalign}{
	\begin{equation}
	\begin{aligned}
}{
	\end{aligned}
	\end{equation}
	\ignorespacesafterend
}

\newcommand{\cl}[1]{{#1}}
\newcommand{\ii}[1]{{#1}}

\title[\cl{nonlocal Kondo model}] 
{Pattern formation in \cl{nonlocal Kondo model}}
\author[Szymon Cygan]{}
\subjclass{35B36, 35Q92, 92C15.}
\keywords{Kondo model, patterns formation, convolution operators, Rabinowitz bifurcation theorem, Schauder fixed point theorem, numerical simulations}

\email{Szymon.Cygan@math.uni.wroc.pl}

\thanks{This work was partially supported by the NCN grant 2016/23/B/ST1/00434.}

\thanks{The author is grateful to Grzegorz Karch for his comments on this work.}

\thanks{This work does not have any conflicts of interest}

\begin{document}
	
\maketitle

\centerline{\scshape Szymon Cygan}
\medskip
{\footnotesize
	\centerline{Instytut Matematyczny, Uniwersytet Wrocławski}
	\centerline{pl. Grunwaldzki 2/4}
	\centerline{ Wrocław 50-384, Poland}
} 

\bigskip

\begin{abstract}
	We study a nonlocal evolution equation generalising a model introduced \cl{by Shigeru Kondo} to explain colour patterns on a skin of \cl{a} guppy fish. We prove \cl{the existence} of stationary solutions using either the bifurcation theory or the Schauder fixed point theorem. We also present numerical studies of this model and show that it exhibits patterns similar to those modelled by well-known reaction-diffusion equations.
\end{abstract}

\section{Introduction}
\label{chap:Introduction}

The goal of this work is to study analytically and numerically stationary solutions to the following initial value problem for an unknown scalar function $u=u(x,t)$
\begin{nalign}
	\label{eq:ProblemFormulation}
	&u_t = -au + f(Tu), && x\in \Omega, \quad t > 0, \\
	&u(x,0) = u_0(x), && x\in \Omega,
\end{nalign}
where $\Omega\subset \R^N$ is bounded, $a>0$ and $T$ is an integral operator. We assume that $f\in C^2$ in the neighbourhood of 0, $f(0) = 0$ and $f'(0) \neq 0$. This is a \cl{minor} generalisation of a model introduced by Shigeru Kondo which we recall and discuss in Subsection~\ref{chap:BiologicalMotivation}, below. 


\cl{Models as those in \eqref{eq:ProblemFormulation} have been introduced in the works  \cite{BR0165423,BR0784129} and the corresponding solutions have been studied numerically only. Our first goal is to prove rigorous mathematical results which we introduce in Section \ref{chap:MainResults}.}
We prove \cl{the existence} of stationary solutions to problem \eqref{eq:ProblemFormulation} using two methods. First, we apply the bifurcation theory to prove \cl{the existence} of \textit{{small}} nonconstant stationary solutions. Next, we construct \textit{{large}} nonconstant stationary solutions using the Schauder fixed point theorem \cl{(see Remark \ref{LargeSmall} for the definitions of \ii{large} and \ii{small} solutions)}. \cl{Then,} in Section~\ref{chap:NumericalSimulations} we present numerical simulations of solutions to model \eqref{eq:ProblemFormulation} and we discuss those numericals results obtained for \cl{a} various parameter range. Numerical simulations indicate that under specific conditions we may obtain diversified patterns, namely nonconstant stationary solutions. In Section~\ref{chap:MathematicalResults} we prove mathematical results stated in Section \ref{chap:MainResults}.

\subsection{Kondo model}
\label{chap:BiologicalMotivation}
The mathematical model proposed by Kondo \cite{BR0784129} is a nonlocal differential equation of the form \eqref{eq:ProblemFormulation} where $u = u(x,t)$ describes the concentration of a specific substance in a fish skin, where a fish surface is a bounded and connected subset of plane $\Omega \subset \R^2$. In that model, the substance production rate results from destruction law and cell synthesis law. The \cl{degradation} law states that production rate is negatively impacted by the substance density. Cell synthesis law states that production rate is influenced by the substance distribution over the surface. This phenomenon is described by the differential equation 
\begin{nalign}
	\label{eq:KondoRawModel}
	\frac{\partial{u}}{\partial t} = S(u)-au,
\end{nalign}
where $a>0$ is a constant cell \cl{degradation} rate and $S$ corresponds to a cell synthesis. A cell synthesis is a process of sophisticated cell interactions, dependent on stimulation operator. Kondo claimed that cell influence on neighbours production rate depends only on distance between cells. Thus, a cell synthesis is modelled as a convolution with a radial kernel $K$
\begin{nalign}
	\label{eq:KondoRawModelKernel}
	Stim(x,y) = \iint u(x-\xi, y - \eta) K(\sqrt{\xi^2 + \eta^2}) \text{d}\xi \text{d}\eta.
\end{nalign}
To ensure that the cell density \cl{is bounded}, the cell synthesis follows the saturation law, which in the work by Kondo is given by the following formula
\begin{nalign}
	S = \begin{cases}
		0, & Stim < 0, \\
		Stim, &  0<Stim \leqslant M, \\
		M, & Stim > M.
	\end{cases}
\end{nalign} 
This saturation function states that the impact of cell density on a production rate is smaller then $M$ and cannot be negative. In biological models considered by Kondo, the kernel $K = K(r)$ is designed to have the \cl{both positive and negative part, which means} that cells can either increase or decrease neighbours production rate. The positive impact is called activation and negative is called inhibition.

\subsection{Other nonlocal models}
\label{chap:AlternativeModels} 
The nonlocal models with convolution kernels are widely used in various fields such as genetics, neurology and ecology. For example, Amari \cite{MR0681526} modelled the dynamics of neuron fields in the brain using the following equation
\begin{align*}
\tau u_t = -u + w* H(u) +  s, \quad \text{for} \quad x\in \R,
\end{align*}
where $u(x,t)$ is the membrane potential of the neurons, $w$ is the convolution kernel, $s$ describes the external stimuli and $H$ is the Heaviside function. The convolution operator represents the influence of cells in the neighbourhood on the membrane potential. 

The following equation in another model of a nonlocal spatial dispersal 
\begin{align*}
u_t = k*u - bu + f(u),
\end{align*}
where $u(x,t)$ denotes the population density, $k$ is the convolution kernel, $b$ is the positive constant and $f$ is the nonlinear function. The kernel $k$ corresponds to the transition possibility and $b$ describes the degradation rate. This model was studied by Hutson et al. \cite{MR2028048} where, despite the behaviour of this model is similar to the corresponding reaction diffusion system, it is more suitable to describe a single species dispersal.

Berestycki {\it et al.} \cite{MR2557449} analysed the nonlocal Fisher-KPP equation for a population dynamics with nonlocal interactions
\begin{nalign}
u_t &= \Delta u + \mu u (1  - k*u), && x\in \R^d
\end{nalign}
and showed that this model has travelling waves, similarly as is \cl{for} the classical Fisher equation.

Another important mathematical result comes from the work by Ninomiya et al. \cite{MR3694699} who studied an extension of a reaction-diffusion system to the nonlocal evolution equation on the one dimensional torus
\begin{nalign}
&u_t = \Delta u + g(u, J*u), && x\in \mathbb{T}, \quad t > 0, \\
&u_0(x) = u(x,0), &&  x\in \mathbb{T}. 
\end{nalign}
They proved that under particular conditions, this nonlocal model can be approximated by solutions to classical reaction-diffusion systems.

\cl{In comparison to the models described above, the Kondo equation contains the truncation function applied to the convolution operator. This truncation plays the crucial role in the process of pattern formation and this phenomenon was not studied in the case of other models.}


\section{Main results}
\label{chap:MainResults}
In this work, we consider the following initial value problem
\begin{nalign}
	\label{eq:MainResultProblemDefinition}
	&u_t = -au + f(Tu), && x\in \Omega, \quad t > 0, \\
	&u(x,0) = u_0(x), && x\in \Omega.
\end{nalign}
with an unknown scalar function $u = u(x,t)$. Here $\Omega \subset \R^\DN$ is a bounded open set, $a>0$ is a constant and function $f:\R\to \R$ is Lipschitz. Nonlocal effects in this equation are described by a linear operator $T:L^2(\Omega) \to L^2(\Omega)$ given by the formula  
\begin{nalign}
	\label{eq:OperatorTDefinition}
	Tu(x) = \int_\Omega K(x-y) u(y) \dy
\end{nalign}
with $K\in L^2(\R^\DN)$ satisfying $K(z) = K(-z)$ for all $z\in \R^\DN$. 

\begin{remark}
		There exists a unique global-in-time solution $u\in C([0,\infty), L^2(\Omega))$ of problem \eqref{eq:MainResultProblemDefinition} for each initial condition $u_0 \in L^2(\Omega)$ because $f$ is assumed to be a globally Lipschitz function. This is a standard result involving the Banach contraction principle. Moreover one can easily prove that if $|f(x)| \leqslant M$ for all $x\in \R$, then $|u(x,t)| \leqslant \max\left\lbrace \|u_0\|_{\infty},\, \frac{M}{a} \right\rbrace$ for all $(x,t) \in \Omega \times [0,\,\infty)$.
\end{remark}

\begin{remark}
	\label{rem:OrthonormalBasis}
Since $\OS$ is bounded, the operator $T : L^2(\Omega) \to L^2(\Omega)$ given by equation \eqref{eq:OperatorTDefinition} is symmetric and compact. By the spectral theory for such operators, there exists a sequence of eigenvalues $\lbrace \lambda_j \rbrace_{j=1}^\infty \subset \R$ satisfying $\lambda_j \to 0$ and an orthonormal basis $\lbrace e_j \rbrace_{j=1}^\infty$ of $L^2(\Omega)$ of eigenfunctions of the operator $T$, namely, we have $Te_j = \lambda_j e_j$ for every $j\in \N$.
\end{remark}

First, we consider the linear counterpart of problem \eqref{eq:MainResultProblemDefinition}
\begin{nalign}
		\label{eq:MainResultLinearProblemDefinition}
	&u_t = -au + b\cdot Tu, & x\in \Omega, &&  t > 0, \\
	&u(x,0) = u_0(x),  & x\in \Omega,
\end{nalign}
with $b \in \R\setminus \lbrace 0 \rbrace$. Notice that, if $\frac{a}{b} = \lambda_k$ for some $k\in \N$, then the eigenfunction $e_k$ corresponding to eigenvalue $\lambda_k$ is a nonconstant stationary solution of the first equation in \eqref{eq:MainResultLinearProblemDefinition}. In the following proposition we find a condition under which this stationary solution is stable.
\begin{proposition}
	\label{thm:KondoModelLinearizedStability}
	Let $\su = e_k$ be a nonzero stationary solution of linear problem \eqref{eq:MainResultLinearProblemDefinition} where $ e_k$ is an eigenfunction of the operator $T$ corresponding to the eigenvalue $\lambda_k = \frac{a}{b}$. The solution $\su$ is stable if and only if $b\lambda_j \leqslant a$ for each $j\in \N$.
\end{proposition}

\noindent
The proof of this proposition is postponed to Section \ref{chap:LinearizedModel}. Now, we only notice that, by Proposition~\ref{thm:KondoModelLinearizedStability} if $b>0$ then $\su = e_k$ is stable if and only if $\lambda_k = \frac{a}{b}$ is the biggest eigenvalue of $T$. On the other hand, if $b<0$ then $\su = e_k$ is stable if and only if $\lambda_k = \frac{a}{b}$ is the smallest eigenvalue of $T$.

Next, we study stationary solutions of the nonlinear problem \eqref{eq:MainResultProblemDefinition}, namely we consider the nonlinear and nonlocal equation
\begin{nalign}
	\label{def:StationarySolutionBifurcation}
	0 = -a\su + f(T\su), \quad x\in \Omega
\end{nalign}
with an unknown function $\su = \su(x)$ and $\Omega \subset \R^\DN$. 

First, we prove \cl{the existence} of solutions to this equation in a neighbourhood of zero solution by using \cl{the} bifurcation theory. Since, we apply methods from the theory of elliptic equations and variational methods we need to impose additional assumptions for the eigenvalues of the operator~$T$. 
\begin{assumption}
	\label{Ass:NonlinearModelSpectralProperties}
	Let $\lbrace \lambda_j \rbrace_{j=1}^\infty$ be the set of eigenvalues from Remark~\ref{rem:OrthonormalBasis}. We assume that $\lambda_j \neq 0$ for each $j\in \N$ and $\lambda_j<0$ for a finite number of eigenvalues. 
\end{assumption}

\begin{theorem}
	Let Assumption \ref{Ass:NonlinearModelSpectralProperties} be satisfied and denote by $d>0$ an arbitrary constant such that $d+a\lambda_j >0$ for each $j\in \N$. Assume $f\in C_b^2$ satisfies $f(0) = 0$ and $\lambda_k = \frac{a}{f'(0)}$ for some $k\in \N$. There exists a sequence $\lbrace c_n\rbrace_{n=1}^\infty\subset \R$ converging to $1$ and a sequence of nonconstant functions $\lbrace \su_n \rbrace_{n=1}^\infty \subset L^2(\Omega)$,  such that $(\su_n)$ is a weak solution of 
	\begin{nalign}
		\label{eq:KondoModelNonlinearSequence}
		0 =& -a c_n \su_n  + f(T\su_n) +d(1-c_n)  T\su_n ,
	\end{nalign}
	for each $n\in\N$.
	\label{thm:KondoModelNonlinearExistence} 
\end{theorem}
\noindent
We postpone the proof of this theorem to Section \ref{chap:BifurcationTheorem}, where solutions are constructed by variational methods and the Rabinowitz bifurcation theorem \cite{MR0463990}. Note that nonzero solutions of equation \eqref{eq:KondoModelNonlinearSequence} are obtained \cl{from the} bifurcation of a nonzero solution to the linear equation	$0 = - a \su + f'(0) T\su$.

Nontrivial solutions in Theorem \ref{thm:KondoModelNonlinearExistence} are constructed in a small neighbourhood of zero solution. Now, we construct \ii{large} solutions and for simplicity of the exposition, we consider the one dimensional problem with an open set $\Omega \subset \R$ and the function $f$ given explicitly by  
\begin{nalign}\label{eq:NonlinearModelFdefinition}
	f(x) = \begin{cases}
		1, & x>1,\\
		x, & x\in[-1,\;1],\\
		-1, & x<-1.
	\end{cases} 
\end{nalign}
\cl{
\begin{remark}
	\label{LargeSmall}
	In this work we deal with \textit{large} and \textit{small} solutions. Here, $u$ is called \textit{small} if $\big|T(u)\big|<1$ and for the nonlinearity given by \eqref{eq:NonlinearModelFdefinition} we have $f\big( T(u)\big) = T(u)$. Otherwise the solution is called \textit{large} and the nonlinearity truncates $T(u)$ at the levels $\pm 1$.   
\end{remark}
}
 \begin{assumption}
 	\label{ass:KernelSchauder}
 We assume that the kernel $K(x) = K(|x|)$ is even and compactly supported. We decompose the kernel as follows
 \begin{align*}
 K = K_+ + K_- \quad \text{with} \quad K_+ = \max(K , 0) \quad \text{and} \quad K_- = \min(K,0)
 \end{align*}  
 where
\begin{itemize}	
	\item $\supp(K_+) \subset [-2,\; 2]$,
	\item $\supp(K_-) \subset [-4, \; -2] \cup [2,\; 4]$.
\end{itemize}
\end{assumption}

Examples of kernels with properties in Assumption \ref{ass:KernelSchauder} are presented in Fig \ref{fig:KondoModelSimulationsConvolutionKernels}. In the following theorem we construct \ii{large} solutions to equation \eqref{def:StationarySolutionBifurcation} under additional constraints imposed on the kernel $K$.

\begin{theorem}
	Assume $\Omega = [-L, \, L]$ with $L\geqslant 6$. Let the function $f$ be given by formula~\eqref{eq:NonlinearModelFdefinition}. Let Assumption~\ref{ass:KernelSchauder} holds true. Suppose, moreover, that the kernel $K$ satisfies one of the following conditions
	\begin{enumerate}
		\item either $K(x) \geqslant 0$ for all $x\in \R$ and \label{prop:Positive}
		\begin{align*}\int_\R K(x)\dx \geqslant 2a,
		\end{align*}
		\item or $K_+(x)$ is a nonincreasing function for $x\geqslant 0$ and  \label{prop:Nonincreasing}
		\begin{align*}
		\int_\R K(x)\dx \geqslant 2a,
		\end{align*}
	\end{enumerate}
	Then, there exists a nonconstant solution to equation \eqref{def:StationarySolutionBifurcation}.
	\label{thm:KondoModelSchauderNonNegative}
\end{theorem}
\begin{remark}
	The kernel $K_3$ on Fig. \ref{fig:KondoModelSimulationsConvolutionKernels} has the property \eqref{prop:Positive} in Theorem \ref{thm:KondoModelSchauderNonNegative} and the kernel $K_1$~--~property~\eqref{prop:Nonincreasing}.
\end{remark}
\noindent
For some nonpositive kernels, one can also construct periodic solutions for equation \eqref{def:StationarySolutionBifurcation} on the whole line.
\begin{theorem}
	Let $\Omega =  \R$. Assume that function $f$ is given by formula~\eqref{eq:NonlinearModelFdefinition}. Let Assumption~\ref{ass:KernelSchauder} holds true. If the kernel $K$ satisfies  
	\begin{itemize}
	\item $K(x) \leqslant 0$ for all $x\in \R$,
	\item  $\int_\R K(x)\dx \leqslant -2a$, 
	\item $K(x) = K(6-x)$ for all $x\in [2,\, 4]$,
	\item $K(x) = K(6+x)$ for all $x\in [-4,\, -2]$
	\end{itemize} 
	then there exists a nonconstant periodic solution to equation~\eqref{def:StationarySolutionBifurcation} considered on the whole line $\Omega = \R$.
	\label{thm:KondoModelNegativeKernelsSchauder} 
\end{theorem}
\begin{remark}
	The kernel $K_4$ on Fig. \ref{fig:KondoModelSimulationsConvolutionKernels} satisfies all assumptions {of} Theorem \ref{thm:KondoModelNegativeKernelsSchauder}.
\end{remark}
\noindent
Stationary solutions in Theorem \ref{thm:KondoModelSchauderNonNegative} and Theorem \ref{thm:KondoModelNegativeKernelsSchauder} are obtained via the Schauder fixed point theorem and we postpone the proofs to Section \ref{chap:SchauderTheorem}. 

\begin{remark}
	Equation \eqref{def:StationarySolutionBifurcation} reduces to the Kondo model \cite{BR0784129} in the case of the cut-off function
	\begin{align*}
	f(x) = \begin{cases}
	1, & x>1,\\
	x, & x\in[0,\;1],\\
	0, & x<0.
	\end{cases} 	
	\end{align*}
	Kernels considered in \cite{BR0784129} are always sign changing functions because in the case of either nonpositive or nonnegative kernels, no nonconstant stationary solutions have been observed numerically. In our case of the function $f(x)$ given by \eqref{eq:NonlinearModelFdefinition} we observe \cl{patterns} also for nonpositve and nonnegative kernels.
\end{remark}

\begin{remark}
	Ideas from the proof of Theorem \ref{thm:KondoModelSchauderNonNegative} and Theorem \ref{thm:KondoModelNegativeKernelsSchauder} in one dimensional case can be used to obtain pattern in two dimensions, but this would require more involved assumptions for convolution kernels. Numerical simulations (Fig. \ref{fig:KondoModelSimulations2DKernelK1Pattern} --  \ref{fig:KondoModelSimulations2DKernelK5Pattern} ) indicate that a type of a convolution kernel strongly influences a shape of a transient gap, however this topic requires further investigation. 
\end{remark}


\section{Numerical simulations}
\label{chap:NumericalSimulations}

\subsection{Description of the problem}

Here we illustrate theoretical results from the previous section by presenting numerical simulations of solution to model \eqref{eq:MainResultProblemDefinition} with $a=1$ and with a suitable odd, monotone, nondecreasing function $f$. More precisely, we consider the problem

\begin{nalign}
	\label{eq:KondoModelNonlinearSimulationsDefinition}
	&u_t = -u + f(Tu) & x\in \Omega, && t>0,  \\ 
	&u_0 = u(0,x), & x\in \Omega,	
\end{nalign}
with

\begin{nalign}
	\label{eq:FunctionFDefiniotonNumerical}
	f(x) = \begin{cases}
		-1, & b x < -1, \\
		b x, &  -1 \leqslant b  x \leqslant 1, \\
		1, & b x > 1,
	\end{cases}
\end{nalign} 
and with a parameter $b>0$. 

We consider the one dimensional and two dimensional version of this problem. In one dimension, we choose $\Omega = [-50,\, 50]$ and the convolution kernel satisfying $\supp(K) \subset [-4, \, 4]$. An initial condition $u_0(x)$ is either random with values from the interval $[-1, \, 1]$ or it is given by
\begin{nalign}
	\label{eq:StepInitialCondition1D}
	u_0 = \begin{cases}
		-1, & x\geqslant 0, \\
		1, & x < 0. \\
	\end{cases}
\end{nalign}
 In the two dimensional case, we set $\Omega = [-50,\, 50]^2$ and here we assume that $\supp(K) \subset [-4, \, 4]^2$. Initial condition $u_0(x)$ is either random from interval $[-1, \, 1]$ or it is given by
 \begin{nalign}
 		\label{eq:StepInitialCondition2D}
 	u_0(x) = \begin{cases}
 		1, & x\in [2, \, 2]^2, \\
 		-1, & x\notin [2, \, 2]^2 \\
 	\end{cases}
 \end{nalign}

\subsection{Numerical scheme}
\label{SubSec:NumericalScheme}

We approximate solution to problem \eqref{eq:KondoModelNonlinearSimulationsDefinition} on a grid consisting 601 uniformly distributed points in one dimension and $601\times 601$ points in the two dimensional case. We use the explicit Euler method with a fixed time step $dt>0$
\begin{nalign}
	\label{eq:NumericalScheme}
	u_{n+1}  &= u_n + dt\big(-a u_n + f(\tilde{T}u_n)\big),
\end{nalign}
where $\tilde{T}$ denotes the discrete version of the convolution operator $T$ given by the formula \eqref{eq:OperatorTDefinition}
\begin{nalign}
	\tilde{T}u_n(x) = \sum_{i\in -300}^{300} K(i) u_n(x - i)
\end{nalign}
in one dimension and with the radial extension in two dimensional case.

\subsection{\cl{Summary of numerical observations}}
\cl{
	Let us summarise the most important properties of numerical stationary solutions to problem \eqref{eq:KondoModelNonlinearSimulationsDefinition} obtained via scheme \eqref{eq:NumericalScheme} for large~$n$. We claim that $\overline{u} = u_n$ is a numerical stationary solution to problem \eqref{eq:KondoModelNonlinearSimulationsDefinition} if $\|u_{n} - u_{n-1}\|_\infty \leqslant \eta$ where $\eta=10^{-6}$ is our premised precision. We present numerical stationary solutions obtained for different types of convolutions kernels, shown in Fig. \ref{fig:KondoModelSimulationsConvolutionKernels} and  Fig \ref{fig:KondoModelSimulations2DConvolutionKernels}, two values of parameter $b$ and various types of initial conditions, random or step-like functions \eqref{eq:StepInitialCondition1D}-\eqref{eq:StepInitialCondition2D}.
}

\subsection{One dimensional simulations}
We present numerical simulations of solution to problem~\eqref{eq:KondoModelNonlinearSimulationsDefinition} for kernels $K_1$, $K_2$, $K_3$ and $K_4$ shown on Fig. \ref{fig:KondoModelSimulationsConvolutionKernels}. The kernel $K_1$ has positive part concentrated near zero which corresponds to a local activation. On the other hand, the negative part of $K_1$ describes long range inhibition. The kernel $K_2$ has positive and negative part as well, but its positive part is not supported in the neighbourhood of zero. Here, our numerical simulations show that this property of the kernel modifies a shape of obtained patterns. The kernel $K_3$ has only a positive part, thus the inhibition process is absent. Finally the kernel $K_4$ with only the negative part describes only an inhibition process. Our numerical simulations shows that, even if the kernel $K_4$ does not stimulate $u$ to rise, we still obtain some patterns.

\begin{figure}[!h]
	\begin{center}
		\includegraphics[width=1\textwidth]{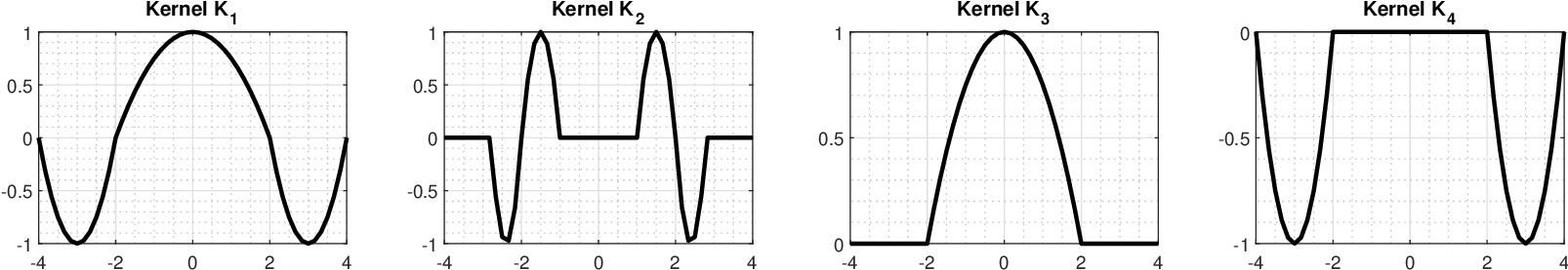}
	\end{center}
	\caption{The kernels $K_1$, $K_2$, $K_3$ and $K_4$ used in one dimensional numerical simulations.}
	\label{fig:KondoModelSimulationsConvolutionKernels}
\end{figure}

We obtain two types of stable stationary solutions. \ii{Small} stationary solutions are of the first type and their existence have been proved in Theorem~\ref{thm:KondoModelNonlinearExistence}. Since these solutions do not touch the truncation level of the function $f$ in \eqref{eq:FunctionFDefiniotonNumerical}, they are solutions to the linear problem considered in Proposition~\ref{thm:KondoModelLinearizedStability}. Fig.~\ref{fig:KondoModelSimulationsLinearEigenfunctions1D} shows such solutions obtained for the kernels $K_1$, $K_2$, $K_3$ and $K_4$. In order to have such solutions, we choose random initial condition in scheme \eqref{SubSec:NumericalScheme} and we set $b = \frac{1}{\lambda_k}$, where $\lambda_k$ is the maximal eigenvalue of the convolution operator~$T$, obtained through numerical approximation.

\begin{figure}[!h]
	\begin{center}
		\includegraphics[width=1\textwidth]{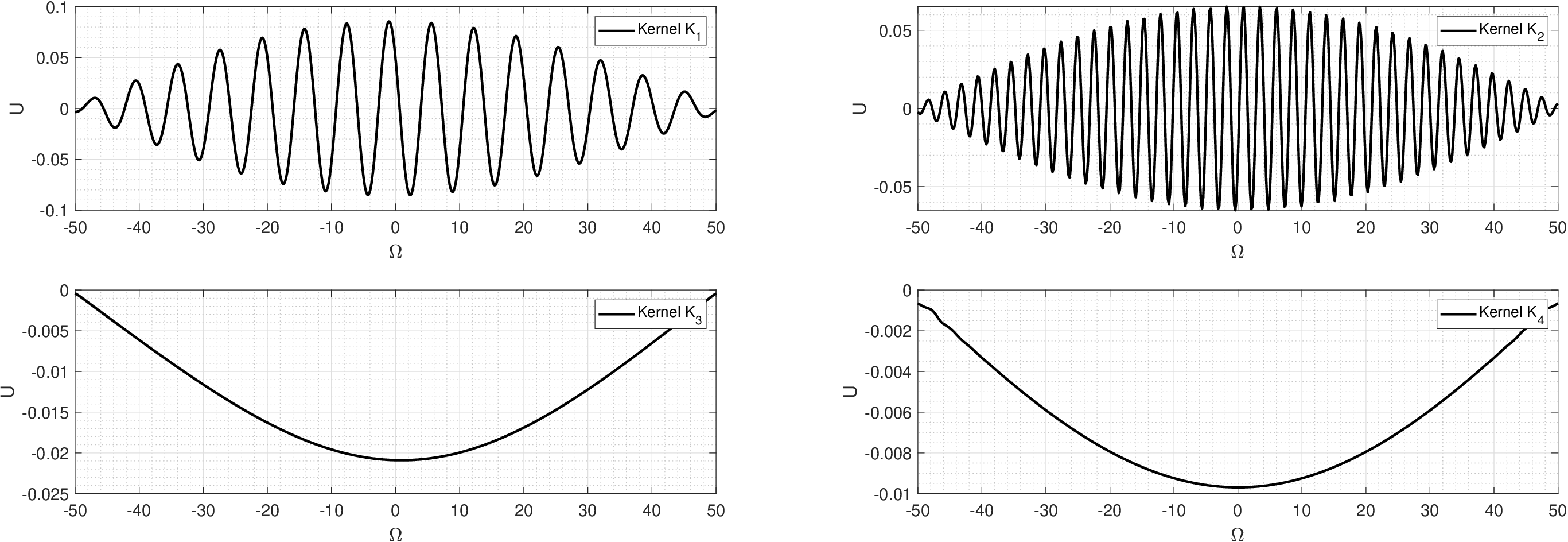}
	\end{center}
	\caption{Stationary solutions to one dimensional problem \eqref{eq:KondoModelNonlinearSimulationsDefinition} obtained via scheme~\eqref{SubSec:NumericalScheme} with random initial conditions. Here we choose $b=0.04035$ for $K_1$, $b=0.08982$ for $K_2$, $b=0.06266$ for $K_3$ and $b=0.063202$ for $K_4$. Since, these solutions do not touch the truncation level in the function $f$, these are, in fact, stationary solutions to linear problem \eqref{eq:MainResultLinearProblemDefinition} with $a=1$.}
	\label{fig:KondoModelSimulationsLinearEigenfunctions1D}
\end{figure}

Next, we obtain \ii{large} stationary solutions considered in Theorem \ref{thm:KondoModelSchauderNonNegative} and Theorem \ref{thm:KondoModelNegativeKernelsSchauder}. The existence of such solutions is strongly connected with the saturation property of the function $f$, because these solutions hit the saturation level $\pm 1$. Fig.~\ref{fig:KondoModelSimulationsSchauderTheorem1D} shows such solutions obtained for the kernels $K_1$, $K_2$, $K_3$ and $K_4$. Here, in the numerical scheme, we take the step-like initial condition $u_0$ given by \eqref{eq:StepInitialCondition1D}. For completeness of the exposition, we present also similar simulations for random initial conditions in Fig.~\ref{fig:KondoModelSimulationsSchauderTheorem1DRandom}.

\begin{remark}
	In Section \ref{chap:SchauderTheorem} we introduce particular family of sets $B$ which are invariant under operator $f(T)$. We claim that the width of the gap required to jump from the level $1$ to $-1$ is at least as big as the support of the positive part of convolution kernel. According to our simulations, the width of gap is significantly smaller then the value assumed during the construction of $B$. Since the shape of function $u$ in the gap vary for different types of kernels, precise construction of set $B$ would require much more technical details.
\end{remark}

\begin{figure}[!h]
	\begin{center}
		\includegraphics[width=1.03\textwidth]{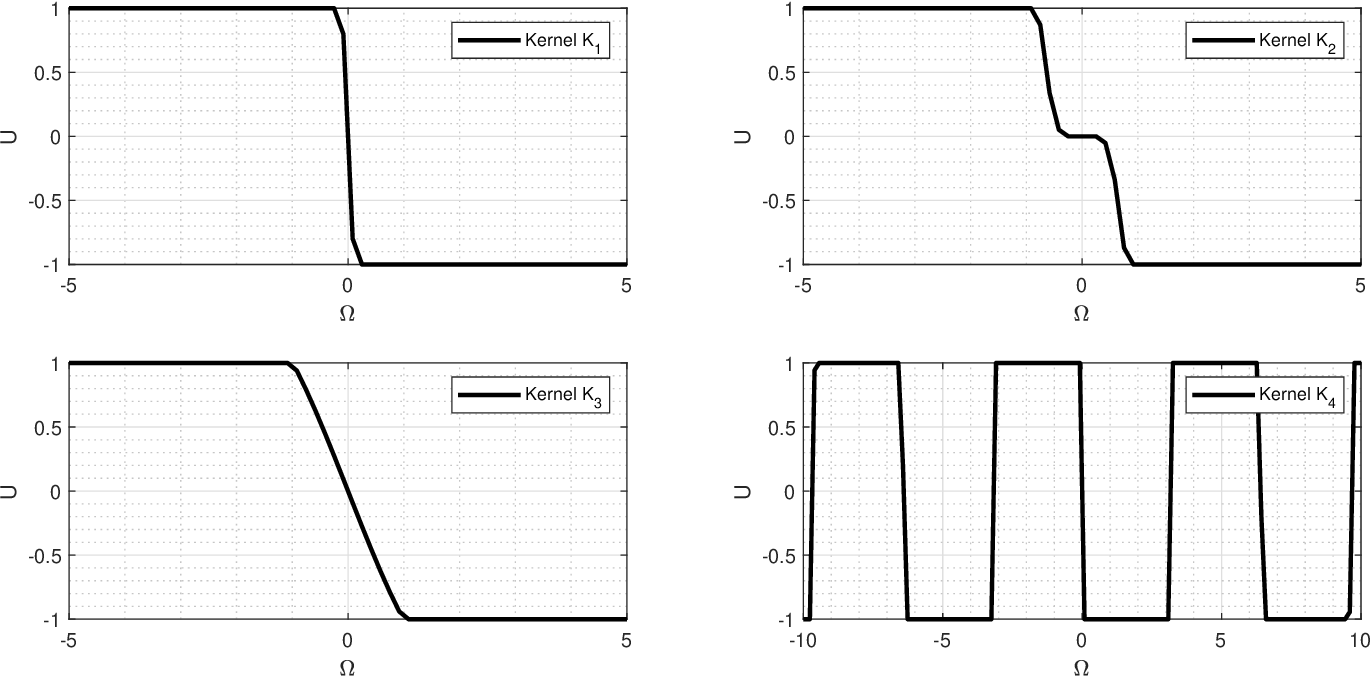}
	\end{center}
	\caption{Numerical stationary solutions to one dimensional problem \eqref{eq:KondoModelNonlinearSimulationsDefinition} obtained via scheme~\eqref{SubSec:NumericalScheme} with initial conditions given by \eqref{eq:StepInitialCondition1D}. Here, $b=0.8$ for $K_1$, $b=0.7$ for $K_2$, $b=0.1$ for $K_3$ and $b=1.0$ for $K_4$. }
	\label{fig:KondoModelSimulationsSchauderTheorem1D}
\end{figure}

\begin{figure}[!h]
	\begin{center}
		\includegraphics[width=1.03\textwidth]{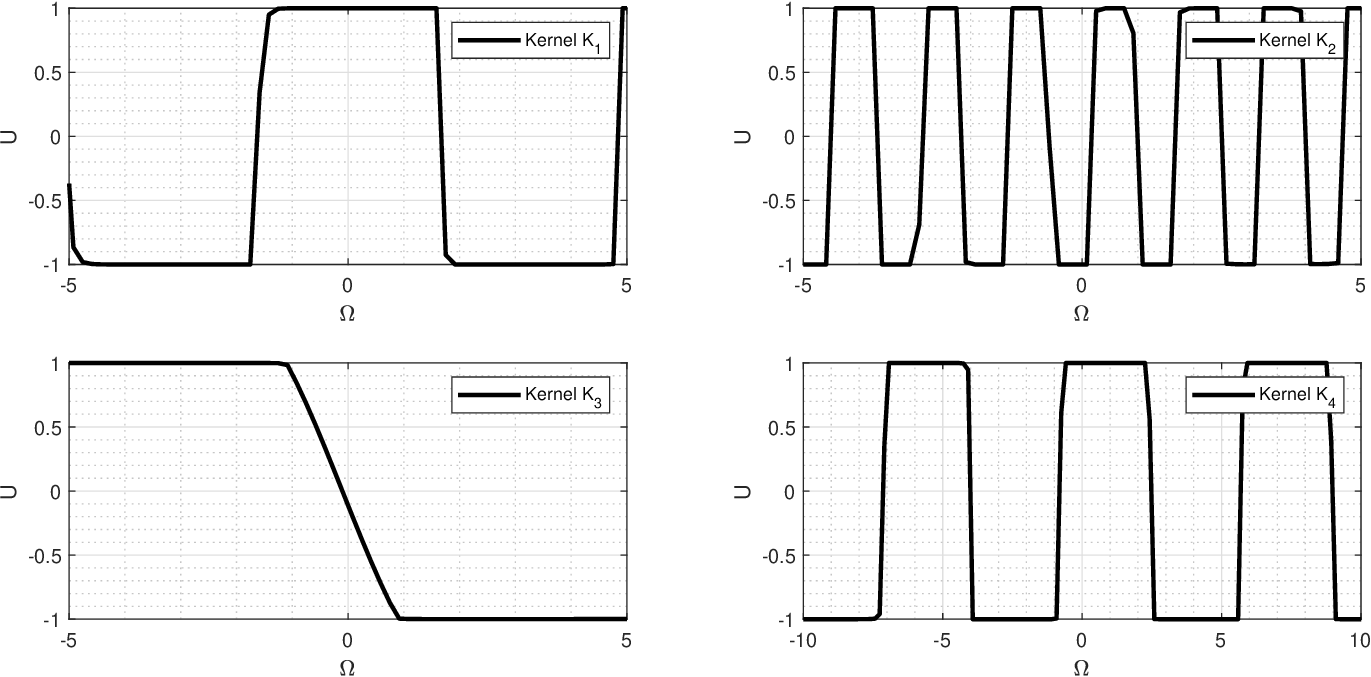}
	\end{center}
	\caption{Numerical stationary solutions to one dimensional problem \eqref{eq:KondoModelNonlinearSimulationsDefinition} obtained via scheme~\eqref{SubSec:NumericalScheme} with random initial conditions. Here, $b=0.8$ for $K_1$, $b=0.7$ for $K_2$, $b=0.1$ for $K_3$ and $b=1.0$ for $K_4$.}
	\label{fig:KondoModelSimulationsSchauderTheorem1DRandom}
\end{figure}

\newpage

\subsection{Two dimensional results}
Here, we present two dimensional stable stationary solutions obtained via scheme \eqref{SubSec:NumericalScheme} and we focus only on \ii{large} stationary solutions. We present simulations for five different kernels $K_5$, $K_6$, $K_7$, $K_8$ and $K_9$ shown in Fig. \ref{fig:KondoModelSimulations2DConvolutionKernels}. The kernels $K_5$ and $K_6$ are positive near zero, and negative otherwise, what corresponds to the local activation and the long range inhibition. Both kernel have the same positive part, but the kernel $K_6$ corresponds to the weaker inhibition effect. Kernel $K_7$ is zero near $0$, positive at some distance from zero, and negative otherwise. We have long range activation and long range inhibition. The support of the negative part is set to to ensure that the integral over the kernel is relatively small. Kernel $K_8$ is negative in the neighbourhood of 0 and positive otherwise, which corresponds to the local inhibition and long range activations. Kernel $K_9$ has only negative part.

\begin{figure}[!h]
	\begin{center}
		\includegraphics[width=0.9\textwidth]{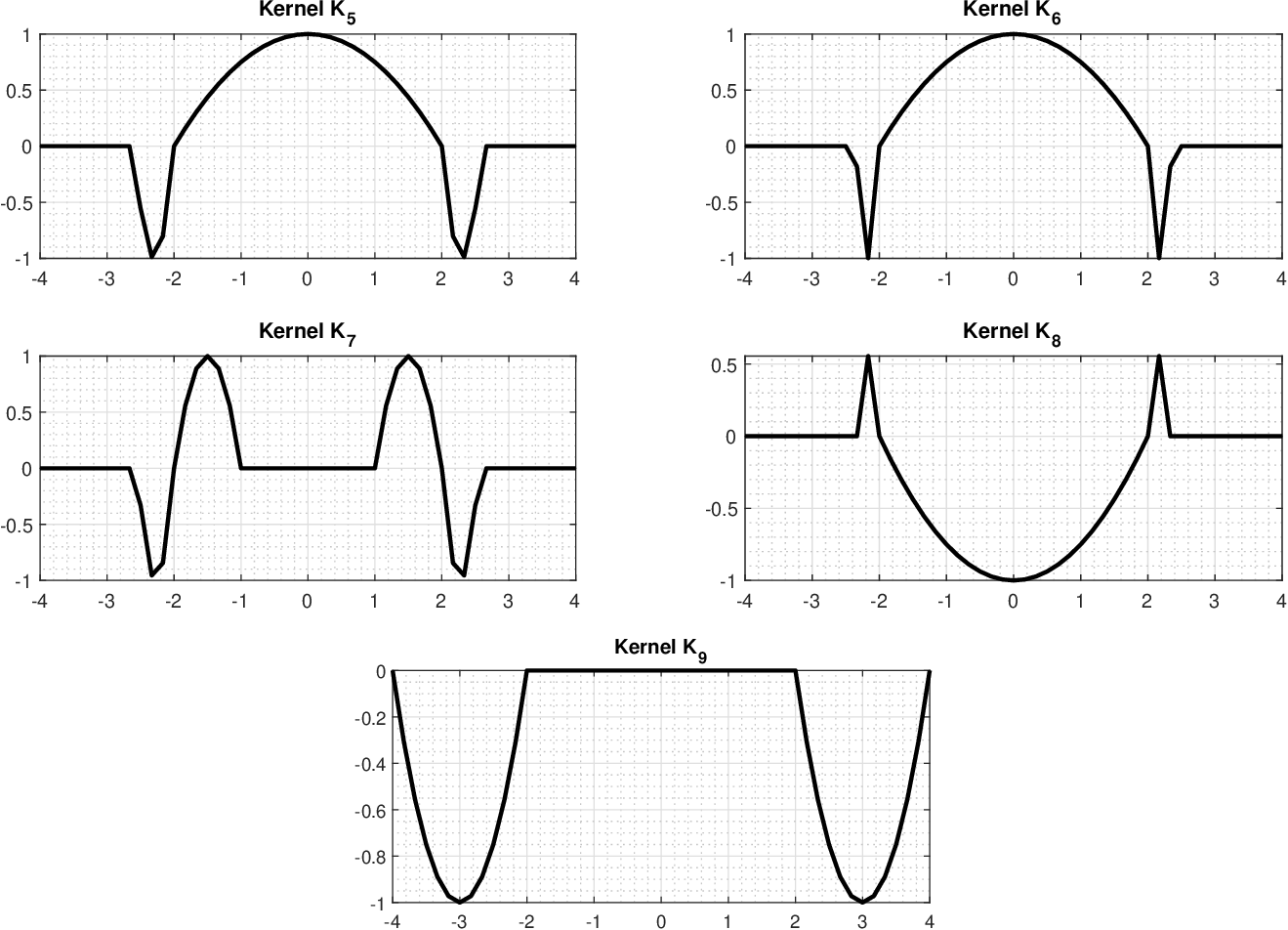}
	\end{center}
	\caption{The kernels $K_5$, $K_6$, $K_7$, $K_8$ and $K_9$ used in two dimensional numerical simulations.}
	\label{fig:KondoModelSimulations2DConvolutionKernels}
\end{figure}

For each kernel, we present four patterns which are obtained from either step-like initial condition~\eqref{eq:StepInitialCondition2D} or random initial condition and two different values of $b$, either $b$ slightly greater then $b_{ctirical}$ (denoted $b = b_{ctirical} + \varepsilon$ with sufficiently small $\varepsilon >0$) or $b$ significantly greater then $b_{critical}$ (denoted $b>>b_{critical}$). There exists a particular constant, namely $b_{critical} = \frac{1}{\lambda_k}$, where $\lambda_k$ is the maximal eigenvalue of operator~$T$, such that for each $b < b_{critical}$ every solution of problem \eqref{eq:KondoModelNonlinearSimulationsDefinition} with arbitrary initial condition $u_0$ converges to $0$. When $b = b_{critical}$ we have linear case and obtained patterns are eigenfunctions of operator $T$. To obtain \ii{large} nonconstant stationary solutions  we need to ensure that at least one eigenfunction becomes unstable, namely $b > b_{critical}$.

\begin{figure}[!t]
	\begin{center}
		\includegraphics[width=1\textwidth]{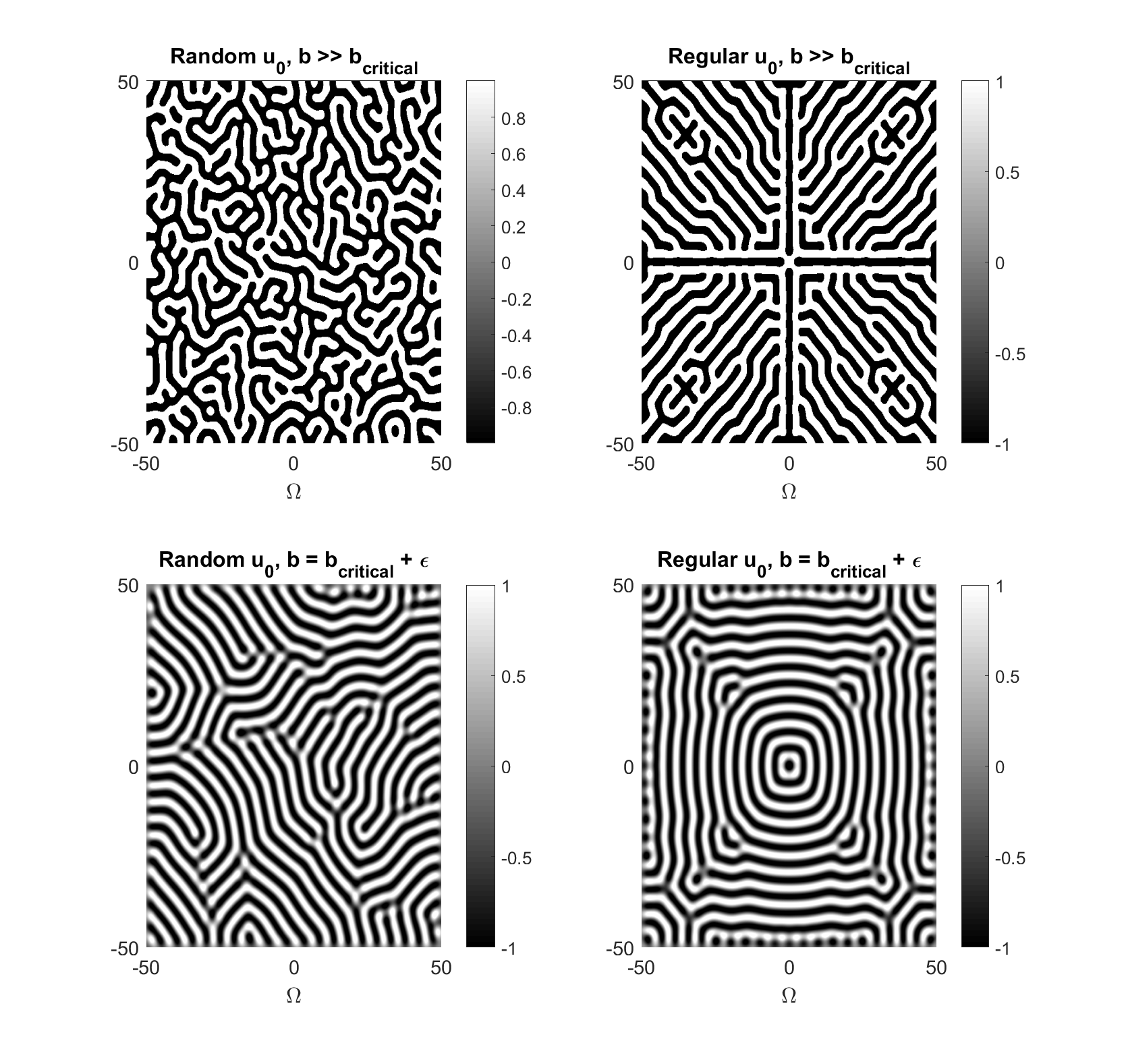}
	\end{center}
	\caption{Numerical stationary solutions to problem \eqref{eq:KondoModelNonlinearSimulationsDefinition} obtained via scheme~\eqref{SubSec:NumericalScheme} for kernel $K_5$. Here, $b = 1$ (upper row) and $b = 0.0063$ (lower row), initial condition is random (left column) or step-like function \eqref{eq:StepInitialCondition2D} (right column). We obtain symmetric patterns for symmetric initial condition \eqref{eq:StepInitialCondition2D}. The width of black and white paths is constant and it depends on the the size of the kernel support. If $b = b_{critical} + \varepsilon$ we have only a few unstable eigenfunctions, hence patterns are more regular. However, if $b>>b_{critical}$, then there are many unstable eigenfunctions and patterns are more diversified and involved.}
	\label{fig:KondoModelSimulations2DKernelK1Pattern}
\end{figure}
\clearpage

\begin{figure}[!t]
	\begin{center}
		\includegraphics[width=1\textwidth]{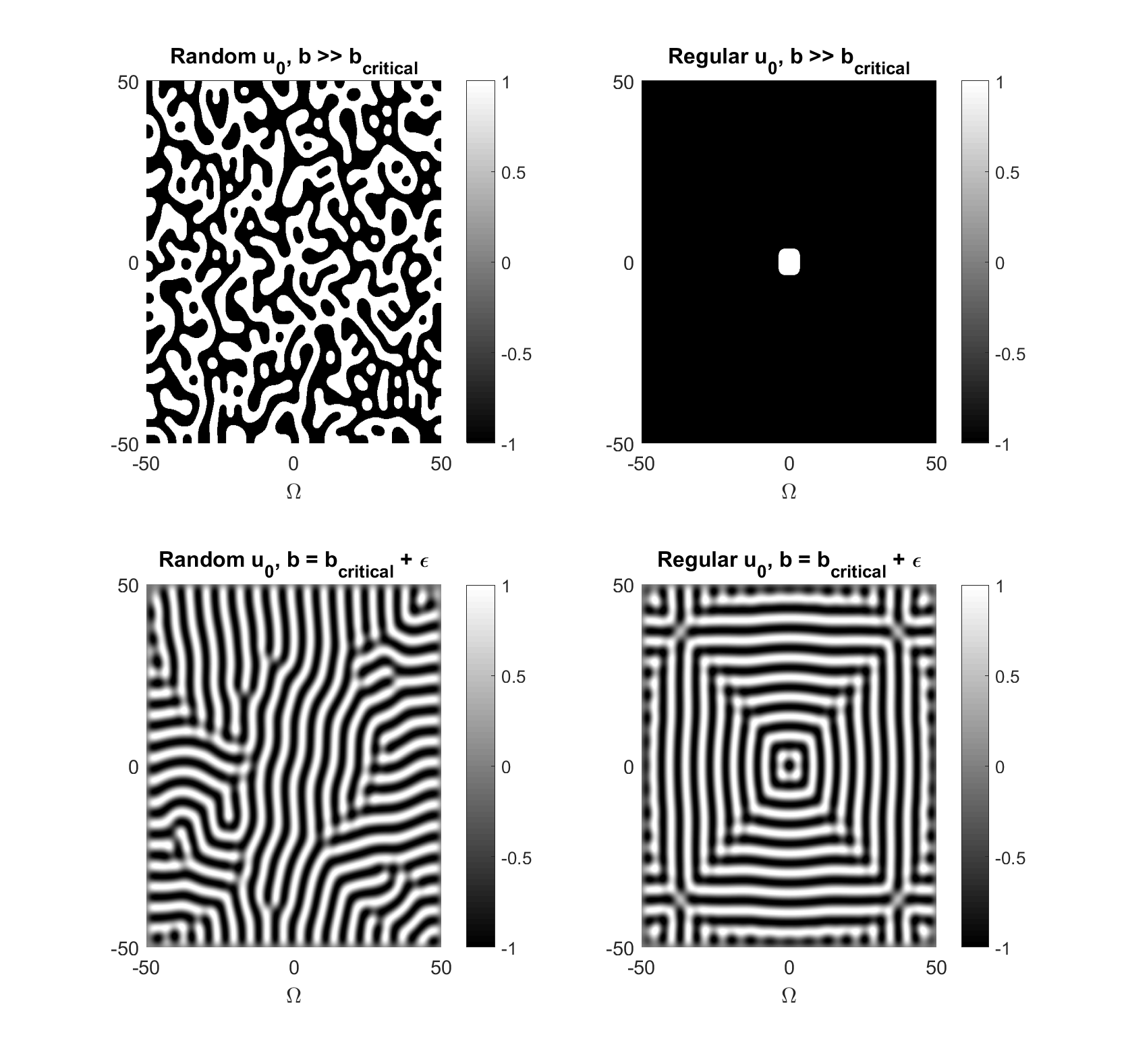}
	\end{center}
	\caption{Numerical stationary solutions to problem \eqref{eq:KondoModelNonlinearSimulationsDefinition} obtained via scheme~\eqref{SubSec:NumericalScheme} for kernel $K_6$. Here, $b = 1$ (upper row) and $b = 0.0070$ (lower row), initial condition is random (left column) or step-like function \eqref{eq:StepInitialCondition2D} (right column). We obtain symmetric patterns for symmetric initial condition \eqref{eq:StepInitialCondition2D}. If $b = b_{critical} + \varepsilon$ we have only a few unstable eigenfunctions. For $b>>b_{critical}$ we have many unstable eigenfunctions. Notice that for large $b$ the width of white and black paths is no longer constant. The negative part of kernel $K_6$ is small and hence large constant areas are invariant under operator $T$. In case of step-like initial condition  \eqref{eq:StepInitialCondition2D}, a large positive part of the kernel is insufficient to spread the positive solution across the domain but it can preserve the existing one if the initial square is sufficiently large.}
	\label{fig:KondoModelSimulations2DKernelK2Pattern}
\end{figure}
\clearpage

\begin{figure}[!t]
	\begin{center}
		\includegraphics[width=1\textwidth]{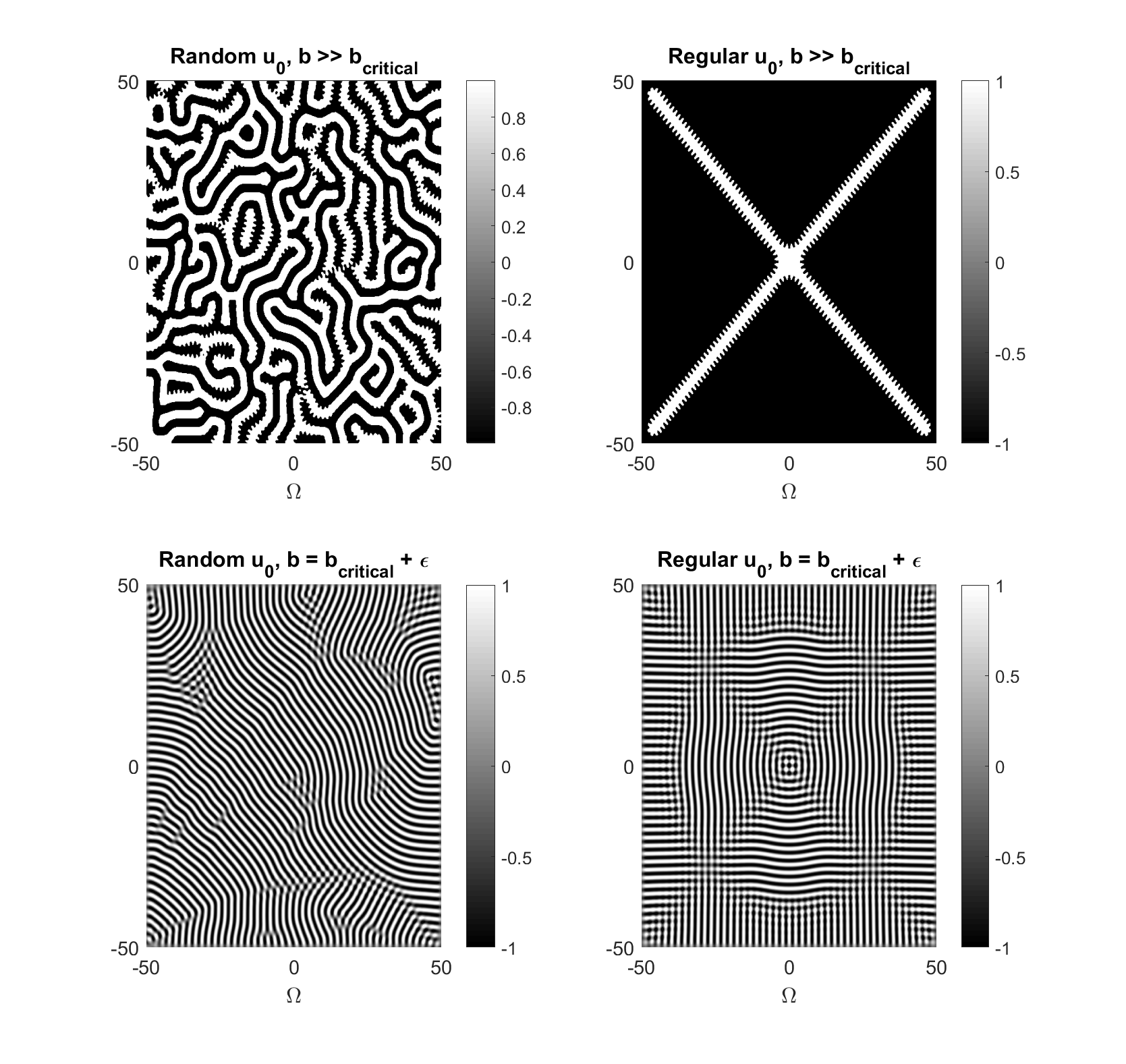}
	\end{center}
	\caption{Numerical stationary solutions to problem \eqref{eq:KondoModelNonlinearSimulationsDefinition} obtained via scheme~\eqref{SubSec:NumericalScheme} for kernel $K_7$. Here, $b = 1$ (upper row) and $b = 0.0127$ (lower row), initial condition is random (left column) or step-like function \eqref{eq:StepInitialCondition2D} (right column). We obtain symmetric patterns for symmetric initial condition \eqref{eq:StepInitialCondition2D}. If $b = b_{critical} + \varepsilon$ we have only a few unstable eigenfunctions. The maximal eigenvalue corresponds to high order eigenfunction and hence the white and black lines are thin. If $b>>b_{critical}$ there are many unstable eigenfunctions. Notice that, locally the solution is distorted with a set of small humps distributed across the white paths. The size of those humps depends on the size of the gap in kernel $K_7$. 
	}
	\label{fig:KondoModelSimulations2DKernelK3Pattern}
\end{figure}
\clearpage

\begin{figure}[!t]
	\begin{center}
		\includegraphics[width=1\textwidth]{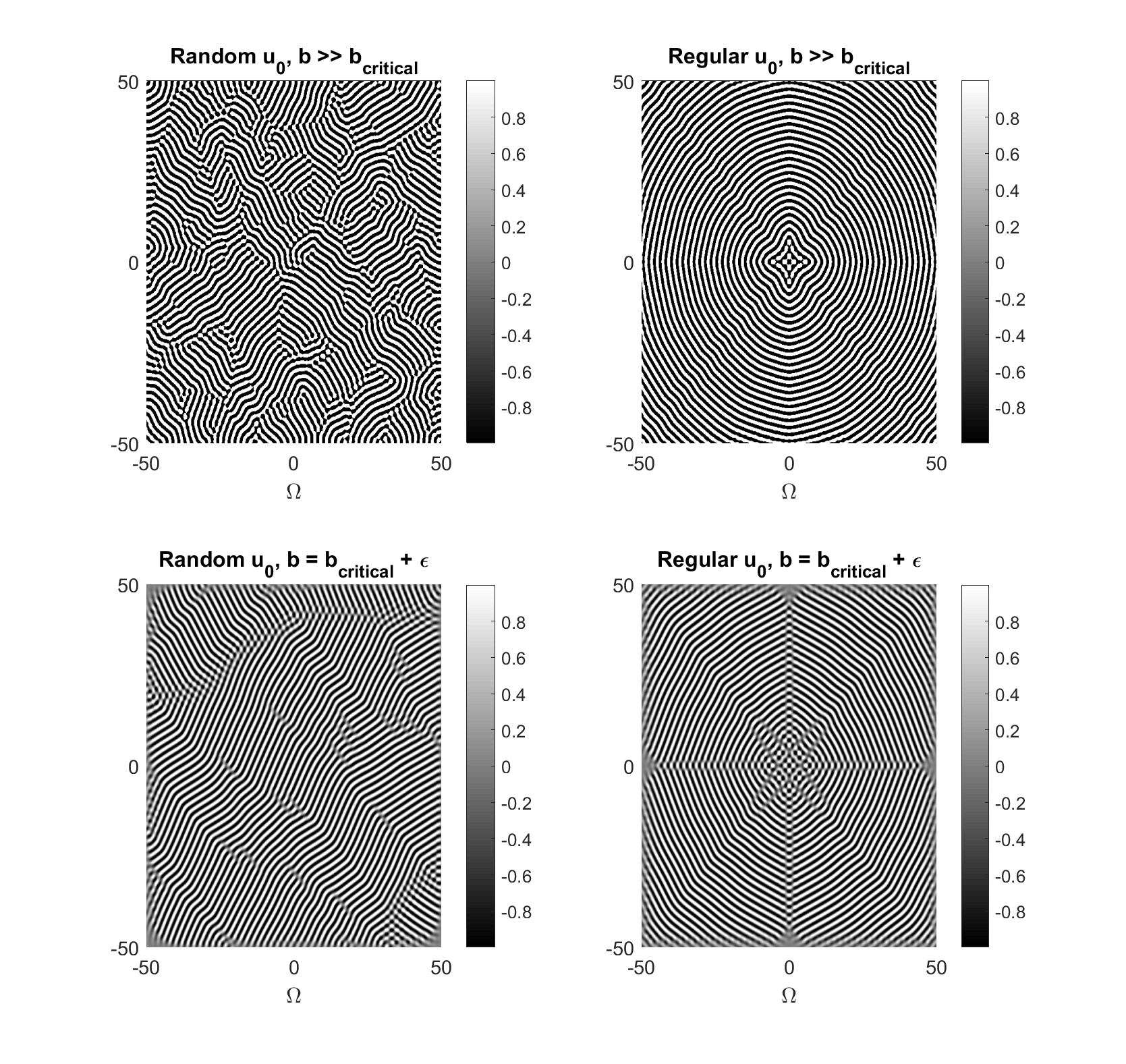}
	\end{center}
	\caption{Numerical stationary solutions to problem \eqref{eq:KondoModelNonlinearSimulationsDefinition} obtained via scheme~\eqref{SubSec:NumericalScheme} for kernel $K_8$. Here, $b = 1$ (upper row) and $b = 0.0352$ (lower row), initial condition is random (left column) or step-like function \eqref{eq:StepInitialCondition2D} (right column). We obtain symmetric patterns for symmetric initial condition \eqref{eq:StepInitialCondition2D}. The width of black and white paths is constant and it depends on the the size of the kernel support. If $b = b_{critical} + \varepsilon$ we have only a few unstable eigenfunctions. If $b >> b_{critical}$ there are many unstable eigenfunctions, however solutions are similar to those obtained for small $b$. There are no other eigenvalues in the neighbourhood of the maximal eigenvalue hence eigenfunction corresponding to maximal eigenvalue dominates the shape of patterns.}
	\label{fig:KondoModelSimulations2DKernelK4Pattern}
\end{figure}
\clearpage

\begin{figure}[!t]
	\begin{center}
		\includegraphics[width=1\textwidth]{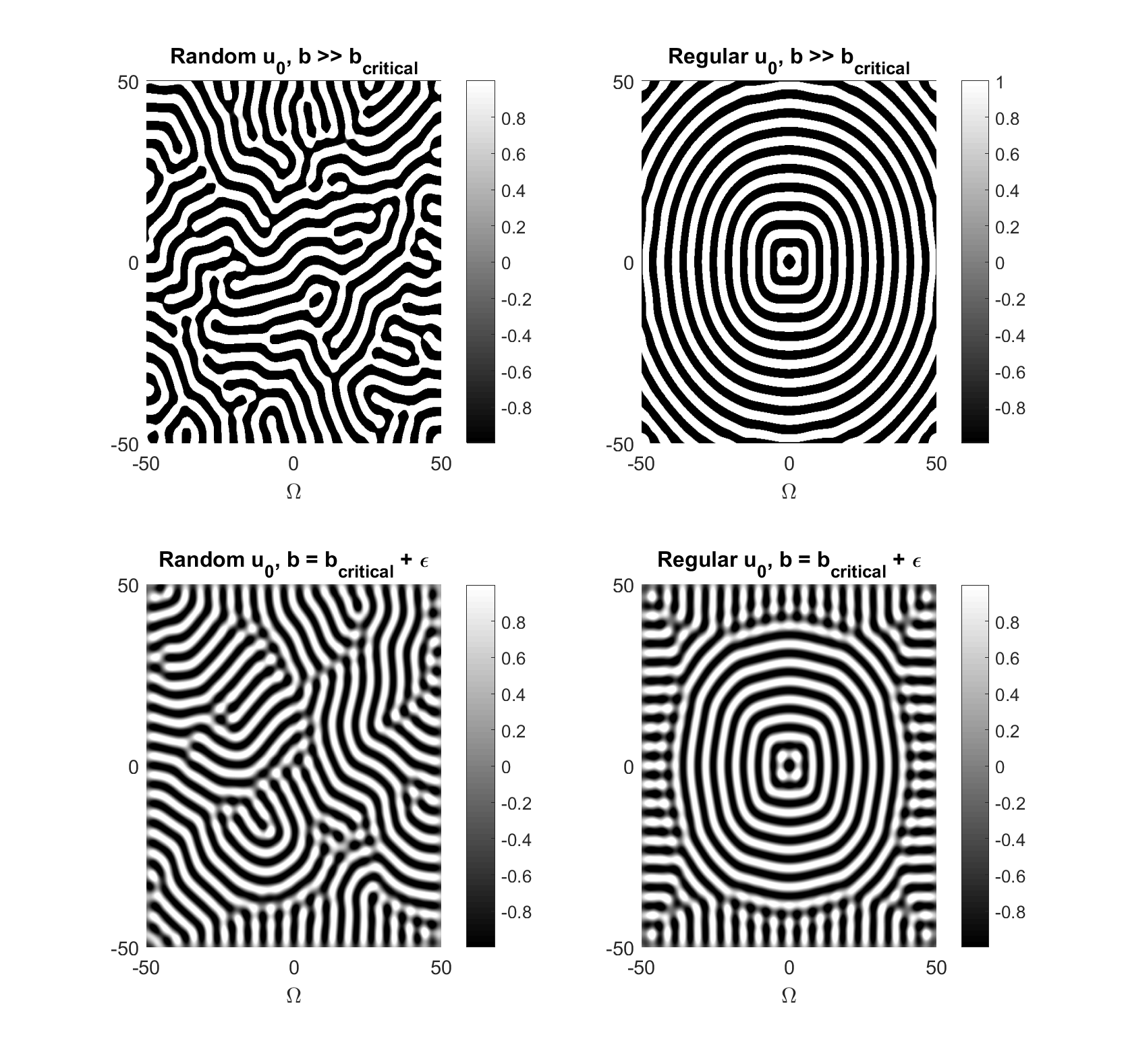}
	\end{center}
	\caption{Numerical stationary solutions to problem \eqref{eq:KondoModelNonlinearSimulationsDefinition} obtained via scheme~\eqref{SubSec:NumericalScheme} for kernel $K_9$. Here, $b = 1$ (upper row) and $b = 0.0040$ (lower row), initial condition is random (left column) or step-like function \eqref{eq:StepInitialCondition2D} (right column). We obtain symmetric patterns for symmetric initial condition \eqref{eq:StepInitialCondition2D}. The width of black and white paths is constant and it depends on the the size of the kernel support. If $b = b_{critical} + \varepsilon$ we have only a few unstable eigenfunctions, hence patterns are more regular. However, if $b>>b_{critical}$, then there are many unstable eigenfunctions and patterns are more diversified and involved.}
	\label{fig:KondoModelSimulations2DKernelK5Pattern}
\end{figure}
\clearpage


\section{Proofs of mathematical results}
\label{chap:MathematicalResults}
\subsection{Stability of solutions to the linear problem}
\label{chap:LinearizedModel}

\begin{proof}[Proof of Proposition \ref{thm:KondoModelLinearizedStability}]
	Let $\su = e_k$ be a stationary solution to the linear problem \eqref{eq:MainResultLinearProblemDefinition}. We consider this problem with perturbed initial condition $u_0 = e_k + u_\varepsilon$ and expand the solution in the orthonormal basis $\lbrace e_j \rbrace_{j=1}^\infty$ of operator $T$ from Remark \ref{rem:OrthonormalBasis}, namely $u(t,x) =  \sum_{j=1}^\infty a_j(t) {e}_j(x)$ with unknown functions $\lbrace a_j(t) \rbrace_{j=1}^\infty$. We substitute $u$ in equation \eqref{eq:MainResultLinearProblemDefinition} and obtain
	\begin{nalign}
		\label{eq:KondoModelLinearizedStabilityMatrixFormDisturbed}
		\sum_{j=1}^\infty a'_j(t) {e}_j(x) & = -a \sum_{j=1}^\infty a_j(t) {e}_j(x) + b\cdot\sum_{j=1}^\infty a_j(t)\lambda_j {e}_j(x)
	\end{nalign}
	which is satisfied if, and only if $\label{eq:LinearModelCoefficientsEquation} 	a_j'(t) = a_j(t)\left( b\cdot\lambda_j - a \right)$ for each $j\in \N$ and hence	\begin{align*}
		\|u(\cdot,t)\|_2 = \sum_{j=1}^\infty a_j(0)^2e^{2(b \lambda_j - a ) t}. 
	\end{align*}
	If $b\cdot \lambda_j - a \leqslant 0$ for each $j \in \N$ then the stationary solution to problem \eqref{eq:MainResultLinearProblemDefinition} is stable and satisfies $\|u(\cdot,t)\|_2 \leqslant \|u_0\|_2$. Since $\lambda_j \to 0$ as $j \to \infty$, then the solution is stable if $\frac{a}{b}$ is equal to the maximal eigenvalue for $b>0$ or minimal eigenvalue for $b<0$.
\end{proof}

%
%

\subsection{Weak solutions to the nonlinear problem}
\label{chap:NonlinearModel}

Under Assumption \ref{Ass:NonlinearModelSpectralProperties} we may define, formally the operator $T^{-1}$  
\begin{align*}
T^{-1} e_j = \frac{1}{\lambda_j} e_j \text{ for each } j\in \N 
\end{align*}
and rewrite the equation \eqref{def:StationarySolutionBifurcation} in the form
\begin{nalign}
	0 = -a T^{-1}(\overline{v}) + f(\overline{v}).
\end{nalign}
We are going to find a solution to the equivalent equation
\begin{nalign}
	aT^{-1}(\overline{v}) + d\overline{v} = d\overline{v} + f(\overline{v}),
\end{nalign}	
where a fixed parameter $d$ satisfies $d+a\lambda_j > 0 $ for each $j\in \N$. First, we introduce the bilinear form  and weak solution to problem \eqref{def:StationarySolutionBifurcation}
\begin{definition}[Bilinear form $\big(Q,D(Q)\big)$]
	Let Assumption \ref{Ass:NonlinearModelSpectralProperties} hold true. For each $u,\, v \in L^2(\Omega)$ such that $u = \sum_{j=1}^\infty a_j e_j$ and $v= \sum_{j=1}^\infty b_j e_j$ let
	\label{def:BilinearFormDefinition}
	\begin{align*}
	D(Q) &=  \Big\lbrace u = \sum_{j=1}^\infty a_j e_j:\sum_{j=1}^\infty a_j^2 \left(\frac{a}{\lambda_j} + d\right) < \infty \Big\rbrace
	\end{align*}
	and
	\begin{align*}	
	Q(u,v) &= \sum_{j=1}^\infty a_jb_j\left(\frac{a}{\lambda_j} + d\right),\quad \text{for } u,  \overline{v}\in D(Q).
	\end{align*}		
\end{definition} 

\begin{definition}[Weak solution]
	Let Assumption \ref{Ass:NonlinearModelSpectralProperties} hold true.
	The function $\overline{v} \in D(Q)$ is a weak solution of equation \eqref{def:StationarySolutionBifurcation} if 
	\begin{nalign}
		\label{eq:NonlinearWeakSolution}
		\langle \overline{v},\varphi \rangle_Q = d\int_\Omega \overline{v}(x) \varphi(x)\dx + \int_\Omega f\big(\overline{v}(x)\big)  \varphi(x) \dx && \text{for each } \varphi \in D(Q). 
	\end{nalign}
\end{definition}
\noindent

\begin{remark}
	Note that,
	\begin{align*}
	\langle e_j,v\rangle_Q = b_j \Big(\frac{a}{\lambda_j} + d \Big) = \Big(\frac{a}{\lambda_j} + d \Big) \langle e_j, v \rangle = \Big\langle (aT^{-1} +d)e_j, v \Big\rangle.
	\end{align*}
\end{remark}
\begin{remark}
	Since $\frac{a}{\lambda_j} + d >0$ for each $j\in N$ the bilinear form $\langle \cdot, \cdot \rangle_Q$ is a scalar product on $D(Q)$. We denote by $\| \cdot \|_Q$ the corresponding norm.
\end{remark}

\begin{lemma}
	\label{thm:NonlinearModelDomainProperties}
	The image of operator $T$ given by \eqref{eq:OperatorTDefinition} is a subset of $L^\infty(\Omega)$. Moreover ${\rm Im}(T)$ is dense in $D(Q)$ and $D(Q)$ is dense in $L^2(\Omega)$.
\end{lemma}
\begin{proof}
	Let $u\in L^2(\Omega)$. We have
	\begin{align*}
	\|T u\|_\infty &=  \Big\|\int_\Omega {K}(x-y) {u}(y) \dy\Big\|_\infty \leqslant \big\| K(x-\cdot) \big\|_2 \big\|u\big\|_2. 
	\end{align*}
	To show the density, observe that finite sums $u_N = \sum_{j=1}^N a_j e_j \in {\rm Im}(T)$, thus
	\begin{align*}
	\|u - u_N\|_Q = \sum_{j=N+1}^\infty a_j^2 \left(\frac{a}{\lambda_j} + d \right) \xrightarrow{N\to \infty} 0 && \text{and} && \|u - u_N\|_2 = \sum_{j=N+1}^\infty a_j^2 \xrightarrow{N\to \infty} 0.
	\end{align*} 
\end{proof}
\begin{remark}
	\label{rmk:NormInequality}
	Notice that $\|\cdot \|_{L^2} \leqslant C \|\cdot \|_Q$ for all $u\in D(Q)$ which is a consequence of the relation $\lambda _j \to 0$.
\end{remark}

We study problem \eqref{eq:NonlinearWeakSolution} \cl{with} variational methods. For a function $F'(v) = f(v)$ we define a functional $J :D(Q) \to \R$, by the formula
\begin{nalign}
	\label{eq:KondoModelNonlinearFunctional}
	J(v) &= \frac{d}{2}\int_\Omega v(x)^2\dx + \int_\Omega F\big(v(x)\big)\dx, & \text{for each } v\in D(Q).
\end{nalign} 
First, we prove basic properties of the functional $J$.

\begin{lemma}
	For every $f\in C^2(\R)$ the functional 
	\begin{nalign}
		J(v) &= \frac{d}{2}\int_\Omega v(x)^2\dx + \int_\Omega F\big(v(x)\big)\dx
	\end{nalign}
	satisfies $J \in C^2(D(Q),\R)$.
	\label{thm:C2functional}
\end{lemma}
\noindent
We skip a direct proof of this Lemma.

\noindent

\subsection{Existence of stationary solutions using a bifurcation theorem}
\label{chap:BifurcationTheorem}

Let $X$ be a real Hilbert space, $\Omega \subseteq X$ be a neighbourhood of 0. Let $L:\Omega \to X$ be a linear continuous operator and let $H \in C(\Omega, X)$. Set $H(v) = o(\|v\|)$ as $v\to 0$. Consider the abstract equation 
\begin{nalign}
	Lv + H(v) = \lambda v.
	\label{eq:RabinowitzBiffurcationTheoremEquation}
\end{nalign}
Obviously, there exists a trivial solution $(\lambda, 0)\in \R \times X$ for each $\lambda$.
\begin{definition}
	A point $(\mu,0) \in \R \times X$ is called a bifurcation point for equation \eqref{eq:RabinowitzBiffurcationTheoremEquation} if every neighbourhood of $(\mu,0)$ contains a nontrivial solution of \eqref{eq:RabinowitzBiffurcationTheoremEquation}.
\end{definition}

\begin{remark}
	Notice that, if $(\mu,0)$ is a bifurcation point then $\mu$ belongs to the spectrum of operator~$L$.
\end{remark}

\begin{proof}
	Let $(\mu,0)$ be a bifurcation point for equation \eqref{eq:RabinowitzBiffurcationTheoremEquation}. Consider the family of balls \\ $B_n\left((\mu,0),\frac{1}{n}\right) \subset \R \times X$. For each $n$ there exists $(\mu_n, {v}_n) \in B_n$ satisfying
	\begin{nalign}
		L{v}_n + H({v}_n) = \mu_n {v}_n.
	\end{nalign}
	Obviously, we have $\mu_n \to \mu$. Divide both sides of this equation by the norm of $\overline{v}_n$ and consider the weak solution
	\begin{equation*}
	\int_\Omega L\frac{L{v}_n}{\|{v}_n\|}\varphi + \int_\Omega \frac{H({v}_n)}{\|{v}_n\|}\varphi = \int_\Omega\mu_n \frac{{v}_n}{\|{v}_n\|}\varphi.
	\end{equation*}
	By assumption for $H$, if $n\to \infty$ we obtain $\int \frac{H({v}_n)}{\|{v}_n\|}\varphi\to 0$. Put ${w}_n = \frac{{v}_n}{\|{v}_n\|}$. Sequence $\lbrace {w}_n\rbrace_{n=1}^\infty$ is bounded, hence it is weakly compact. Hence, ${w}_{n_k} \rightharpoonup {w}$ and
	\begin{equation*}
	\int_\Omega L({w}) \varphi = \int_\Omega \mu {w} \varphi.
	\end{equation*}
	The equality holds for each $v \in X$, hence $\mu$ belongs to the spectrum of the operator $L$.
\end{proof}

\noindent
Now, we recall a classical result from the bifurcation theory.

\begin{theorem}
	[Rabinowitz Bifurcation Theorem, \cite{MR0463990}]
	Let $X$ be a real Hilbert space, $U$ a neighbourhood of 0 in $X$ and $I \in C^2(U, \R)$ with $I' (v) = Lv + H(v)$, $L$ be linear and $H(v) = o(\|v\|)$ at $v = 0$. If $\mu$ is an isolated eigenvalue of $L$ of finite multiplicity, then $(\mu,0)$ is a bifurcation point for \eqref{eq:RabinowitzBiffurcationTheoremEquation}. Moreover, at least one of the following occurs:
	\label{thm:RabinowitzBiffurcationTheorem}
	\begin{enumerate}
		\item $(\mu, 0)$ is and isolated solution of \eqref{eq:RabinowitzBiffurcationTheoremEquation} in $\lbrace \mu \rbrace \times X$ 
		\item  There is one-sided neighbourhood, $\Lambda$ of $\mu$ such that for all $\lambda \in \Lambda \setminus \lbrace \mu \rbrace$, equation \eqref{eq:RabinowitzBiffurcationTheoremEquation} possesses at least two distinct nontrivial solutions.
		\item There is a neighbourhood $I$ of $\mu$ such that for all $\lambda \in I\setminus \lbrace \mu \rbrace$, equation \eqref{eq:RabinowitzBiffurcationTheoremEquation} possesses at least one nontrivial solution.
	\end{enumerate}
\end{theorem}

\noindent
We are now ready to prove the existence of stationary nontrivial solutions for problem \eqref{eq:NonlinearWeakSolution}. 

\begin{proof}[Proof of Theorem \ref{thm:KondoModelNonlinearExistence}]
We define operators $L$ and $H$
\begin{equation*}
L({v}) = \Big(\frac{a}{\lambda_j} + d \Big){v}  \quad \text{and} \quad H({v}) =  f({v}) - a\frac{1}{\lambda_j} {v}.
\end{equation*}
The equation \eqref{eq:NonlinearWeakSolution} can be rewritten as
\begin{nalign}
	\label{eq:BifurcationForm}
	\langle \overline{v}, \varphi \rangle_Q = \int_\Omega L(\overline{v})\varphi + H(\overline{v})\varphi  \quad \text{for} \quad \varphi \in D(Q).
\end{nalign}
Let us define the nonlinear functional $I(v):D(B) \to \R$
\begin{nalign}
	I(v) = \frac{\frac{a}{\lambda_k} + d}{2} \int_\Omega v^2 + \int_\Omega \left(F(v) - \frac{a}{2\lambda_k}  \cdot v^2 \right). 
	\label{eq:KondoModelNonlinearFunctionalBifurcation}
\end{nalign}

	Let $I(v)$ be as in \eqref{eq:KondoModelNonlinearFunctionalBifurcation} then $DI(v) = Lv + H(v)$. From Lemma \ref{thm:C2functional} we obtain that functional $I\in C^2(D(Q),\R)$. To prove that $H(u) = o(\|v\|)_{D(Q)}$, we need to check that if $\|v\|_{D(Q)} \to 0$, then
	\begin{nalign}
		\frac{\|H(v)\|_{D(Q)}}{\|v\|_{D(Q)}} \to 0.
	\end{nalign}
	Since $f(0) = 0$ we obtain
	\begin{nalign}
		\frac{\|f(v) - a\frac{1}{\lambda_k} \cdot v\|_{D(Q)}}{\|v\|_{D(Q)}} = \frac{\| f(v) -  f(0) - a\frac{1}{\lambda_k} v\|_{D(Q)}}{\|v\|_{D(Q)}}.
	\end{nalign}
	The right hand side tends to the derivative of $f$ if $\|v\|_{D(Q)} \to 0$. It follows from the assumption that the numerator tends to $0$. 
	
	Now we apply Theorem \ref{thm:RabinowitzBiffurcationTheorem} to obtain that $(1,0)$ is a bifurcation point for system \eqref{eq:BifurcationForm}. It follows that there exist a sequence of $\lbrace c_n\rbrace_{n=1}^\infty$ convergent to $1$ and a sequence of nonconstant functions $\lbrace v_n\rbrace_{n=1}^\infty\subset D(Q)$ such that
	\begin{nalign}
		\label{eq:KondoModelNonlinearStationaryReducedBiffurcated}
		c_n \langle \overline{v}_n,\varphi \rangle_Q = \int_\Omega \left(\frac{a}{\lambda_j} + d \right) \overline{v}_n \varphi + \int_\Omega \left(f(\overline{v}_n) - \frac{a}{\lambda_j} \overline{v}_n \right) \varphi
	\end{nalign}
	for each $n\in \N$ and each $\varphi \in D(Q)$. The equation \eqref{eq:KondoModelNonlinearStationaryReducedBiffurcated} is a weak solution of 
	\begin{nalign}
		0 =& -a d_n \overline{u}_n  + f(T\overline{u}_n) +d(1-d_n) T\overline{u}_n.
	\end{nalign}
\end{proof}

\subsection{Existence of solutions using the Schauder fixed point theorem}
\label{chap:SchauderTheorem}

  There exists a solution to equation \eqref{def:StationarySolutionBifurcation} if $\overline{u}$ is a fixed point of the compact operator
\begin{nalign}
	\label{eq:MappingDefinition}
	\frac{1}{a}f\big(T(\cdot)\big): L^2(\Omega) \to L^2(\Omega)
\end{nalign}
which we obtain from the Schauder fixed point theorem. Obviously $\overline{u}\equiv 0$ is a fixed point of this operator in the case of $f$ given by formula \eqref{eq:FunctionFDefiniotonNumerical}. To ensure the existence of nonconstant solutions, we construct particular invariant sets $B$ for this mapping which do not contain constant functions. We introduce appropriate sets for three main classes of convolution kernels: nonnegative kernels and sign changing kernels in Theorem \ref{thm:KondoModelSchauderNonNegative} and nonpositive kernels in Theorem \ref{thm:KondoModelNegativeKernelsSchauder}.

\begin{proof}[Proof of Theorem \ref{thm:KondoModelSchauderNonNegative}] Let the condition \eqref{prop:Positive} in Theorem \ref{thm:KondoModelSchauderNonNegative} hold true. We introduce 
	\begin{nalign}
	\label{eq:KondoModelSchauderB1spaceDefinition1}
	B_1 = \left\lbrace u \in L^2(\Omega): 
	u(x) =
	\begin{cases} 
		 -\frac{1}{a} \; &{\rm for} \; x\geqslant 2,\\ 
		 \frac{1}{a} \;\; &{\rm for} \;\; x\leqslant -2, 
	\end{cases} \;\; u {\rm  \; is\; odd \; and \; monotone} \right\rbrace.
	\end{nalign}
	Let us show that $\frac{1}{a}f(T): B_1 \to B_1$. Indeed, if $x \geqslant 2$ then
	\begin{nalign}
		Tu(x) = \int_\Omega K(x-y)u(y) \dy \leqslant \int_2^L K(x-y)u(y) \dy = -\frac{1}{a}\int_0^2 K(y)\dy \leqslant -1
	\end{nalign}
	and hence $\frac{1}{a} f(Tu)(x) = -\frac{1}{a}$. Analogously we obtain that for $x\leqslant -2$ we have $\frac{1}{a} f(Tu)(x) = \frac{1}{a}$. If $|x| <2$ then the convolution of odd and monotone function with positive and even function is odd and monotone. Thus, mapping \eqref{eq:MappingDefinition} has a fixed point in the set $B_1$.
	
Notice that, in the set $B_1$ the width of the gap between levels $\frac{1}{a}$ and $-\frac{1}{a}$ is equal to the support of the kernel. In the next step, we consider smaller gaps. Under the condition \eqref{prop:Nonincreasing} in Theorem \ref{thm:KondoModelSchauderNonNegative} let 
	\begin{nalign}
	\label{eq:KondoModelSchauderB1spaceDefinition2}
	B_2 = \left\lbrace u \in L^2(\Omega): 
	u(x) =
	\begin{cases} 
		 -\frac{1}{a} \; &{\rm for} \; x\geqslant 1,\\ 
		 \frac{1}{a} \;\; &{\rm for} \;\; x\leqslant-1, 
	\end{cases} \;\; u {\rm  \; is\; odd \; and \; monotone} \right\rbrace.
\end{nalign} 
	We show that $\frac{1}{a}f(T): B_2 \to B_2$. For $1\leqslant x\leqslant 4$ we have
	\begin{align*}
	Tu(x) =  \int_{x}^{L} K(x-y)u(y)\dy + \int_{x-4}^x K(x-y)u(y)\dy \leqslant -1 + \int_{x-4}^x K(x-y)u(y)\dy.
	\end{align*}
	Notice that $\int_{x-4}^x K(x-y)u(y)\dy \leqslant 0$. Indeed, if $ 1 \leqslant x \leqslant 2$ then $u(x-2) \geqslant 0$. Since $K_+$ is nonincreasing for positive arguments and $u$ is monotone we have 
	\begin{nalign}
	\int_{x-2}^x K_{+}(x-y)u(y)\dy \leqslant 0 \text{ and } \int_{x-4}^{x-2} K_-(x-y)u(y)\dy \leqslant u(x-2)\int_{x-4}^{x-2} K_-(x-y)dy \leqslant 0.
	\end{nalign}
	Analogously, if $2< x \leqslant 4$ then $u(x-2) \leqslant 0$ and 
	\begin{nalign}
	\int_{x-4}^x K(x-y)u(y)dy \leqslant u(x-2)\int_{-4}^0 K(y)\dy \leqslant 0.
	\end{nalign}	
	If $4 < x \leqslant L$ then $\int_x^L K(x-y)u(y)\dy \leqslant 0$ and hence
	\begin{nalign}
		Tu(x) = \int_{x-4}^x K(x-y)u(y)\dy + \int_x^L K(x-y)u(y)\dy \leqslant -1.
	\end{nalign}
	Consequently we obtain $\frac{1}{a} f(Tu)(x) = -\frac{1}{a}$. The case $x\leqslant -1$ is proved analogously.
	For $|x|<1$ we have 
	\begin{nalign}
		Tu(x) =&  \int_\Omega K_+(x-y)u(y) \dy - \frac{1}{a} \int_{x +2}^{x+4} K_-(x-y) \dy + \frac{1}{a} \int_{x-4}^{x-2} K_-(x-y) \dy \\ = &\int_\Omega K_+(x-y)u(y) \dy
	\end{nalign}
	and hence the negative part of the kernel can be omitted. The convolution of odd and monotone function with positive function is odd and monotone. By the Schauder fixed point theorem, there exists a fixed point of mapping \eqref{eq:MappingDefinition} in the set $B_2$.
\end{proof}
\begin{remark}
	The functions on Fig. \ref{fig:KondoModelSimulationsSchauderTheorem1D} corresponding to kernels $K_1$, $K_2$, $K_3$ belong to the sets $B_1$ and $B_2$ with $a=1$. 
\end{remark}
\noindent
By the same reasoning we immediately obtain the family of nonconstant solutions. Under the assumptions of Theorem \ref{thm:KondoModelSchauderNonNegative}, we can obtain, following more general result.
\begin{remark}
	Let the assumptions of Theorem \ref{thm:KondoModelSchauderNonNegative} holds true. Let $\lbrace \Omega_j \rbrace_{j=1}^N$ be a family of disjoint intervals satisfying $\Omega_j \subset \Omega$ and  
	\begin{itemize}
		\item ${\rm diam}(\Omega_j) \geqslant M$, for each $j$,
		\item ${\rm dist}(\Omega_j, \Omega_k) \geqslant M$, for each $k \neq j$
	\end{itemize}
	where $M = 4$ if condition \eqref{prop:Positive} holds true and $M=2$ if condition \eqref{prop:Nonincreasing} is satisfied. There exists a stationary solution $\overline{u}$ such that  $\left.\overline{u}\right|_{\Omega_j} = i_j$ for an arbitrary sequence $\lbrace i_j \rbrace_{j=1}^N$ satisfying $i_j = \pm \frac{1}{a}$.
\end{remark}

\begin{proof}[Proof of Theorem \ref{thm:KondoModelNegativeKernelsSchauder} ]
	We introduce set 
	\begin{nalign}
		\label{eq:KondoModelSchauderB1spaceDefinition3}
		B_3 = \left\lbrace u \in L^2(\R): 
		\begin{matrix*}[l]u(x) = 
		\begin{cases} 
			\frac{1}{a} \; &{\rm for} \;  x\in [-1+8k,\;1+8k],\\ 
			 -\frac{1}{a} \;\; &{\rm for} \; x\in [3+8k,\;5+8k], 
		\end{cases} \\
  	 u \text{ is even, satisfies $u(x) = -u(x+4)$ and } \\ 
  	 	\text{monotone on each interval $[1+4k, \; 3+4k]$},
  	 \end{matrix*} \; \; k\in \N
		\right\rbrace
	\end{nalign}
	We show that $\frac{1}{a}f(T): B_3 \to B_3$. If $x\in [-1,\;1]$ (similarly for $x\in [-1+8k,\;1+8k]$), we have
	\begin{nalign}
		\label{eq:NegativeKernelsOperator}
		Tu(x)   = \int_\Omega K(x-y)u(y)\dy &=  \int_{x-4}^{-3} K_-(x-y)u(y) \dy + \int_{-3}^{x-2} K_-(x-y)u(y) \dy \\ &+  \int_{x+2}^{3} K_-(x-y)u(y) \dy + \int_{3}^{x+4} K_-(x-y)u(y) \dy.
	\end{nalign} 
	Function $u$ is constant on the intervals $[x-4,\, -3]$ and $[3,\,x+4]$ hence
	\begin{nalign}
		\int_{x-4}^{-3} K_-(x-y)u(y) \dy + \int_{3}^{x+4} K_-(x-y)u(y) \dy = 1. 
	\end{nalign}
	Notice that the sum of second and third integral in equation \eqref{eq:NegativeKernelsOperator} is nonnegative. Indeed, from the symmetry of $u$ and $K$ we have $\int_1^3 K(y+1)u(y)dy = 0$, hence
	\begin{nalign}
		\int_{-3}^{x-2} &K_-(x-y)u(y) \dy + \int_{x+2}^{3} K_-(x-y)u(y) \dy \\ 
		&= -\int_{1}^{x+2} K_-(x+4-y)u(y) \dy + \int_{x+2}^{3} K_-(x-y)u(y) \dy \\
		& = \int_1^3 \big( K_-(x-y) - K_-(x-2-y) \big)	u(y) \dy \geqslant  0.
	\end{nalign}
	Consequently we obtain $\frac{1}{a} f(Tu)(x) = \frac{1}{a}$. The case $x\in [3 + 8k,\;5+ 8k]$ is proved analogously. If $x\in [1,\; 3]$ (similarly $x\in [1+4k,\, 3+4k]$) then
	\begin{nalign}
		Tu(x) = \int_{x-4}^{x-2} K(x-y) u(y)\dy + \int_{x+2}^{x+4} K(x-y) u(y)\dy.
	\end{nalign} 
	Notice that $u(x)$ is nondecreasing for $x\in [-3,\,1]$ and $x\in[3,\,7]$. Thus, the convolution of monotone function with negative function is nonincreasing. Since $u$ and $K$ are even functions, then $Tu$ is even and satisfies
	\begin{nalign}
		-Tu(x+4)  = -\int_\R K(x - y)u(y-4)\dy = \int_\R K(x - y)u(y)\dy= Tu(x). 
	\end{nalign}
	 By the Schauder fixed point theorem, there exists a fixed point of mapping \eqref{eq:MappingDefinition} in the set $B_3$.     
\end{proof}

\end{document}